\documentclass[11pt,reqno]{amsart}
\usepackage[top=1.2in, bottom=1.2in, left=1.2in, right=1.2in]{geometry}
\usepackage{amsfonts, amssymb, amsmath, mathtools, dsfont}
\usepackage[linktoc=page, colorlinks, linkcolor=blue, citecolor=blue]{hyperref}
\usepackage{enumerate, commath}

\usepackage{graphicx,xcolor, abstract}
\usepackage[font=small]{caption} 
\usepackage{algorithm,algpseudocode}

\newtheorem{theorem}{Theorem}[section]
\newtheorem{proposition}[theorem]{Proposition}

\newtheorem{lemma}[theorem]{Lemma}

\newtheorem{definition}[theorem]{Definition}

\theoremstyle{remark}
\newtheorem{remark}[theorem]{Remark}

\numberwithin{equation}{section}
\DeclareMathOperator{\E}{\mathbb{E}}

\DeclareMathOperator{\diam}{diam}

\DeclareMathOperator{\Lap}{Lap}

\renewcommand{\Pr}[2][]{\mathbb{P}_{#1} \left\{ #2 \rule{0mm}{3mm}\right\}}

\def \Z {\mathbb{Z}}

\def \a {\alpha}
\def \b {\beta}

\def \e {\varepsilon}

\def \l {\lambda}
\def \s {\sigma}

\newcommand{\cA}{\mathcal{A}}

\newcommand{\cF}{\mathcal{F}}
\newcommand{\cX}{\mathcal{X}}
\newcommand{\cY}{\mathcal{Y}}

\newcommand{\cM}{\mathcal {M}}
\newcommand{\dd}{\mathrm{d}}

\newcommand{\add}[1]{{#1}}

\renewcommand{\paragraph}[1]{\subsection*{#1}}

\title{Online Differentially Private Synthetic Data Generation}
\usepackage{times}

\author{Yiyun He}
\address{Department of Mathematics, University of California Irvine, Irvine, CA 92697}
\email{yiyunh4@uci.edu}

\author{Roman Vershynin}
\address{Department of Mathematics, University of California Irvine, Irvine, CA 92697}
\email{rvershyn@uci.edu}

\author{Yizhe Zhu}
\address{Department of Mathematics, University of Southern California, Los Angeles, CA, 90089}
\email{yizhezhu@usc.edu}

\begin{document}

\maketitle


\begin{abstract}%
 We present a polynomial-time algorithm for online differentially private synthetic data generation. For a data stream within the hypercube $[0,1]^d$ and an infinite time horizon, we develop an online algorithm that generates a differentially private synthetic dataset at each time $t$. This algorithm achieves a near-optimal accuracy bound of $O(\log(t)t^{-1/d})$ for $d\geq 2$ and $O(\log^{4.5}(t)t^{-1})$ for $d=1$ in the 1-Wasserstein distance. This result extends the previous work on the continual release model for counting queries to Lipschitz queries. Compared to the offline case, where the entire dataset is available at once \cite{boedihardjo2022private, he2023algorithmically}, our approach requires only an extra polylog factor in the accuracy bound.
\end{abstract}

\section{Introduction}
{Differential} privacy (DP) has emerged as a leading standard for safeguarding privacy in scenarios that involve the analysis of extensive data collections. It aims to protect the information of individual participants within datasets from disclosure. At its core, an algorithm is differentially private if it can produce consistently randomized outcomes for nearly identical datasets. This approach to privacy preservation is gaining traction across various sectors, notably in the implementation of the 2020 US Census \cite{abowd2019census, hawes2020implementing, hauer2021differential} and among technology companies \cite{cormode2018privacy, dwork2019differential}. The scope of differential privacy extends to a wide range of data science applications, including statistical query \cite{mckenna2018optimizing}, regression \cite{chaudhuri2008privacy, su2016differentially}, parameter estimation \cite{duchi2018minimax}, and stochastic gradient descent \cite{song2013stochastic}.

Existing research focuses on developing specialized algorithms for specific tasks constrained by predefined queries. This necessitates a significant level of expertise and often requires the modification of existing algorithms. Addressing these challenges, a promising approach is to create a synthetic dataset that approximates the statistical properties of the original dataset while ensuring differential privacy \cite{hardt2012simple, bellovin2019privacy,wasserman2010statistical}. This approach facilitates subsequent analytical tasks on the synthetic dataset without additional privacy risks.

Despite extensive research in differential privacy, most advancements have focused on scenarios involving a single collection or release of data. However, in reality, datasets frequently accumulate over time, arriving in a continuous stream rather than being available all at once. This is common in various domains, such as tracking COVID-19 statistics, collecting location data from vehicles \cite{kim2021survey}, or internet search and click data\cite{chan2011private}. In these contexts, generating online synthetic data that adheres to differential privacy standards poses significant challenges \cite{bun2023continual, kumar2024privstream}.

\subsection{Continual release model}
One popular model in online differential privacy is the \textit{Continual Release Model}, first studied in \cite{dwork2010differential,chan2011private}. In this model, data points arrive in a streaming fashion, and an online algorithm releases the statistics of the streaming dataset in a differentially private manner. The initial example explored in \cite{dwork2010differential,chan2011private} was for Boolean data streams. At time $t$, a Boolean sequence $x_1,\dots, x_t\in \{0,1\}$ is available, and DP algorithms were developed to release the count $\sum_{i=1}^t x_i$ for each $t\leq T$, where $T$ is the time horizon of the streaming data sequence. The scenario when $T=\infty$ is termed the \textit{infinite time horizon}, where the input data stream is an infinite sequence, and the DP algorithm outputs an infinite sequence.

In the online setting, repeating offline DP-counting algorithms would require an increasing privacy budget over time due to the composition property of differential privacy \cite{dwork2014algorithmic}, thus not being feasible. A seminal contribution of \cite{dwork2010differential,chan2011private} is the Binary Mechanism, which achieves $\text{polylog}(t)$ error while maintaining $\varepsilon$-differential privacy for a finite time horizon $T$. Additionally, \cite{dwork2015pure} improved the accuracy of the Binary Mechanism to $O\left(\log(T)+ \log^{1.5}(n)\right)$ when the Boolean data stream is sparse, i.e., the number of 1's, denoted by $n$, is much smaller than $T$. The dependence on $T$ in \cite{dwork2015pure} is optimal and matches the $\Omega(\log T)$ lower bound in \cite{dwork2010differential} for online DP-count release.

The Binary Mechanism serves as a foundational element for many online private optimization problems \cite{guha2013nearly,kairouz2021practical}. Various methods to enhance the Binary Mechanism in different settings have been studied in \cite{chan2011private,dwork2015pure,qiu2022differential,fichtenberger2022constant,henzinger2023almost,henzinger2024unifying}. Besides counting tasks, DP algorithms for online data have also been discussed for mean estimation \cite{george2022continual}, moment statistics \cite{mir2011pan,epasto2023differentially}, graph statistics \cite{song2018differentially}, online convex programming \cite{jain2012differentially}, decaying sums \cite{bolot2013private}, user stream processing \cite{chen2017pegasus}, and histograms \cite{cardoso2022differentially}. Utilizing offline DP algorithms as black boxes, \cite{cummings2018differential} provided a general technique to adapt them to online DP algorithms with utility guarantees.

\subsection{Differentially private synthetic data}

Among all DP algorithms, DP synthetic data generation \cite{abowd2008protective, hardt2012simple, blum2013learning, jordon2019pate, bellovin2019privacy, boedihardjo2022covariance, xin2022federated, fan2022private, liu2024optimal} excels in flexibility because it allows for a wide range of downstream tasks to be performed without incurring additional privacy budgets. However, the computationally efficient construction of DP synthetic data with utility guarantees remains challenging.  \cite{ullman2011pcps} shows that it is impossible to use $\text{poly}(d)$ samples to generate $d$-dimensional DP synthetic data that approximates all 2-way marginals in polynomial time.

Most results for private synthetic data are concerned with counting queries, range queries, or $k$-dimensional marginals \cite{hardt2012simple, ullman2011pcps, blum2013learning, vietriprivate, dwork2015efficient, thaler2012faster, boedihardjo2022private2}, and various metrics have been used to evaluate the utility of DP synthetic data \cite{boedihardjo2022covariance, yang2023differentially, asadi2023gibbs}. A recent line of work \cite{boedihardjo2022private, donhauser2023sample, he2023algorithmically, he2023differentially, donhauser2024privacy} provides utility guarantees with respect to the Wasserstein distance for DP synthetic data. Using the Kantorovich-Rubinstein duality \cite{villani2009optimal}, the 1-Wasserstein distance accuracy bound ensures that all Lipschitz statistics are uniformly preserved. Given that numerous machine learning algorithms are Lipschitz \cite{von2004distance, kovalev2022lipschitz, bubeck2021universal, meunier2022dynamical}, one can expect similar outcomes for the original and synthetic data.



\subsection{Main results}
 We consider the problem of generating DP-synthetic data under the continual release model beyond the Boolean data setting considered in \cite{dwork2010differential,chan2011private}. The data stream comes from the hypercube $[0,1]^d$ equipped with the $\ell_{\infty}$-norm, and our goal is to efficiently generate private synthetic data in an online fashion while maintaining a near-optimal utility bound under the Wasserstein distance. 
 Our main result is given in the next theorem.

\begin{theorem}[Online DP synthetic data]
\label{thm: main}
For any constant $\e>0$, there is an $\e$-differentially private algorithm such that, for any data stream $x_1,\dots,x_t,\dots \in [0,1]^d$, at any time $t$, 
 it transforms the first $t$ points $\cX_t=\{x_1,\dots,x_t\}$ into $t$ points $\cY_t\subset [0,1]^d$, with the following accuracy bound: 
there exists a constant $C_{\e}$ depending only on $\e$ such that for $t\geq C_{\e}$,  
 \begin{align}\label{eq:accuracy_finite}
    \mathbb EW_1(\mu_{\mathcal X_t},\mu_{\mathcal Y_t})\lesssim 
    \begin{cases}
     \log (t) (\e t)^{-\frac{1}{d}}, & d\geq 2,\\
   \log^3(\e t)\log^{1.5}(t)(\e t)^{-1},  & d=1,
    \end{cases}
\end{align}  
where $W_1(\mu_{\mathcal X_t},\mu_{\mathcal Y_t})$ is the 1-Wasserstein distance between two empirical measures $\mu_{\cX_t}, \mu_{\cY_t}$ of  $\cX_t$ and $\cY_t$, respectively.
\end{theorem}

\begin{remark}
 From the proof of  Theorem~\ref{thm: main} in Sections~\ref{sec:proof} and \ref{sec:proof_d_1}, we can choose
 \begin{align}  
   C_{\e}=
    \begin{cases}
    \exp(\log^2(1/\e+1))+e/\e,  & d\geq 2,\\
    e/\e & d=1.
    \end{cases}  \notag
\end{align}  
\end{remark}

Our algorithm is computationally efficient. To obtain  synthetic datasets $\cY_t$ at time $t$, the time complexity is $O(dt+\e t\log t)$; see Section~\ref{sec: running_time} for details. 

The utility guarantee in Theorem~\ref{thm: main} is optimal up to a $\log t$ factor for $d\geq 2$ and $\mathrm{polylog}(t)$ for $d=1$. Compared to offline synthetic data tasks, generating online private synthetic data is much more challenging, especially with an infinite time horizon. For offline private synthetic data on $[0,1]^d$, $d\geq 2$, \cite{he2023algorithmically} proposed an algorithm with utility bound $O(n^{-1/d})$, which matches the minimax lower bound proved in \cite{boedihardjo2022private}.

We prove Theorem~\ref{thm: main} by analyzing our main Algorithm~\ref{alg: online_final} for $d\geq 2$ and $d=1$, respectively. We develop an online hierarchical partition procedure to divide the domain $[0,1]^d$ into disjoint sub-regions with decreasing diameters as time increases and then apply online private counting subroutines to count the number of data points in each subregion. After the online private counting step, we create synthetic data following the Consistency and Output steps described in Algorithm~\ref{alg: online_final}.

A key ingredient in our work is the development of a special \emph{Inhomogeneous Sparse Counting Algorithm} (Algorithm~\ref{alg: inhomogeneous_sparse_couting}) for the online private count of data points in each subregion, which has different privacy budgets for different time intervals. Such dynamic assignments are motivated by the selection of optimal privacy budgets based on the dynamic hierarchical partition. We apply the new counting algorithm with carefully designed privacy parameters and starting times for each subregion based on the hierarchical structure of the online partition.

The concept of counting sparse data also plays an important role. Intuitively, when inputs $x_1,\dots, x_t$ are uniformly distributed in $[0,1]^d$, the online count of a newly created sub-region corresponds to a sum of a sparser Boolean data stream as the diameter of the sub-region decreases. In fact, the uniformly distributed data represents the worst-case configuration of the true dataset in the minimax lower bound proof in \cite{boedihardjo2022private}, corresponding to a sparse Boolean data stream for each subregion with a small diameter. We make use of the sparsity to obtain a near-optimal accuracy bound.

 \add{The main focus of our work is on the theoretical side, proving a theoretical utility upper bound for a computationally efficient online differentially private synthetic data algorithm. The theoretical guarantee is for all deterministic datasets in the hypercube $[0,1]^d$, which should be interpreted as a near-optimal accuracy rate in the minimax sense.  We do not intend to optimize the performance of our algorithm on specific datasets in this work. Other algorithms that are adaptive to specific structures of the data might perform better in practice.}

\paragraph{Comparison to previous results}
Our work generalizes the framework of DP-counting queries for Boolean data in the continual release model \cite{dwork2010differential,chan2011private} to online synthetic data generation in a metric space. One of the subroutines, \textit{Inhomogeneous Sparse Counting} (Algorithm~\ref{alg: inhomogeneous_sparse_couting}), is a generalization of the sparse counting mechanism from \cite{dwork2015pure} with inhomogeneous noise according to the online hierarchical partition structure we introduce. The Binary Mechanism in \cite{dwork2010differential,dwork2015pure} is designed only for a finite time horizon, and a modified Hybrid Mechanism was developed for the infinite case in \cite{chan2011private}. Our Algorithm~\ref{alg: online_final} works for data streams with an infinite time horizon.

In terms of online DP synthetic data generation, \cite{kumar2024privstream} considered online DP-synthetic data for spatial datasets, and \cite{bun2023continual} studied online DP synthetic \add{Boolean data with prefixed counting queries under a different notion of differential privacy called zero-concentrated differential privacy (zCDP). Moreover,
   \cite{bun2023continual} consider specific types of counting queries and the output dataset at time $t$ is obtained from adding one new data point to the output dataset at time $t-1$. This is different from our synthetic data, where we generate different datasets at different time stamps; see part B of Section~\ref{sec: preliminaries} for more details.
To the best of our knowledge, our work is the first to generate online DP-synthetic data with utility guarantees for all Lipschitz queries.}

Finally, we discuss the difference between online DP algorithms and their offline counterparts. \cite{bun2017make} established a separation for the number of offline, online, and adaptive queries subject to differential privacy. For the continual release model, \cite{jain2021price} showed that for certain tasks beyond counting, the accuracy gap between the continual release (online) model and the batch (offline) model is $\tilde{\Omega}(T^{1/3})$, which is much better than the $\Omega(\log T)$ gap shown in \cite{dwork2010differential} between online and offline counting tasks.

Our Theorem~\ref{thm: main} shows that for a dataset in $[0,1]^d$, the accuracy gap between online and offline DP-synthetic data generation is at most a factor of $O(\log(t))$ for $d\geq 2$, and at most $O(\text{polylog}(t))$ for $d=1$. The lower bound in \cite{dwork2010differential} also implies an $\Omega(\log (T))/T$ accuracy lower bound in our setting for online DP synthetic data in $[0,1]$ with time horizon $T$. However, for $d\geq 2$, the argument in \cite{dwork2010differential} cannot be directly generalized to prove an accuracy lower bound for datasets in $[0,1]^d$. We conjecture that when $d\geq 2$, the upper bound in Theorem~\ref{thm: main} is tight in terms of the dependence on $t$.

\paragraph{Organization of the paper}
The rest of the paper is organized as follows. Section~\ref{sec: preliminaries} introduces several notations and definitions. Section~\ref{sec: syn_data} details the online DP-synthetic data generation algorithm. In Section~\ref{sec:proof}, we demonstrate that Algorithm~\ref{alg: online_final} is differentially private with the accuracy bound stated in Theorem~\ref{thm: main} for $d\geq 2$.  The proof of  Theorem~\ref{thm: main} for $d=1$ are provided in Section~\ref{sec:proof_d_1}. \add{Additional proofs are given in  Appendix~\ref{sec:addtional_proofs}.} The analysis of the time complexity for Algorithm~\ref{alg: online_final} is included in Section~\ref{sec: running_time}. 

\section{Preliminaries}

\label{sec: preliminaries}

\subsection{Differential privacy}
We use the following definitions of differential privacy and neighboring datasets from \cite{dwork2014algorithmic}.
\begin{definition}[Neighboring datasets]
    Two sets of data $\mathcal X$ and $\mathcal X'$ are \textit{neighbors} if  $\mathcal X,\mathcal X'$ differ by at most one element. 
\end{definition}
 
\begin{definition}[Differential Privacy] \label{def:DP}
A randomized algorithm $\mathcal A$ is $\e$-differentially private if for any two neighboring data $\cX, \cX'$ and any subset $S$,
\[
\mathbb P \left(\mathcal A(\cX)\in S \right)\leq \exp(\e) \cdot \mathbb P \left(\mathcal A(\cX')\in S \right).
\]
\end{definition}

\subsection{Online synthetic data generation}
We consider DP algorithms to release synthetic data in an online fashion when a data stream arrives. In our setting, a dataset is an infinite sequence \[\cX=(x_1,\dots, x_t,\dots ),\] where each $x_t\in [0,1]^d$ arrives at time $t\in \mathbb Z_+$. Two datasets $\cX, \cX'$ are \textit{neighbors} if they differ in one coordinate.  Define the time-$t$ data stream from $\cX$ as \[\cX_t=(x_1,\dots,x_t).\] 
For each time $t\in \mathbb Z_{+}$, a randomized synthetic data generation algorithm $\cA_{t}$ takes an input $\cX_t$ and outputs a synthetic dataset of size $t$ given by
\[\cY_t=(y_{1,t},\dots, y_{t,t}).\]
 It is not necessary that  $\cY_{t-1}\subset \cY_t$;  they can be completely disjoint.

An \textit{online synthetic data generation algorithm} $\cM$ with infinite time horizon takes an infinite sequence $\cX$ and output an infinite sequence of synthetic datasets such that
\begin{align}
    \cM(\cX):=(\cA_1(\cX_1),\dots, \cA_t(\cX_t), \dots)=(\cY_1,\dots, \cY_t,\dots). \notag
\end{align}
We say $\cM$ is $\e$-differentially private if $\cM$ satisfies Definition \ref{def:DP}, which guarantees that the entire sequence of outputs is insensitive to the change of any individual’s contribution.

\subsection{Wasserstein distance}
Consider two probability measures $\mu,\nu$ in  a metric space $(\Omega,\rho)$. The 1-Wasserstein distance for more details) between them is defined as 
\[W_1(\mu,\nu):=\inf_{\gamma\in\Gamma(\mu,\nu)} \int_{\Omega\times\Omega}\rho(x,y)\dd\gamma(x,y),\]
where $\gamma(\mu,\nu)$ is the set of all couplings  of $\mu$ and $\nu$.  

To quantify the utility of an online DP-synthetic data algorithm $\cA$, we compare the statistical properties of the synthetic data sets $\cY_1,\dots, \cY_t$   with the true datasets $\cX_1,\dots, \cX_t$ for all $t\in \mathbb Z_+$. We identify each data set with an empirical measure and  define 
\begin{align*}
    \mu_{\cX_t}=\frac{1}{t} \sum_{i=1}^t \delta_{x_i}, \quad \mu_{\cY_t}= \frac{1}{t} \sum_{i=1}^{t} \delta_{y_{i,t}}.
\end{align*}
We would like to have the synthetic data stream $(\cY_1,\dots, \cY_t)$ to stay close to the data stream $(\cX_1,\dots, \cX_t)$ under 1-Wasserstein distance for all $t\in \mathbb Z_+$.

For two probability measures $\mu$ and $\nu$ on $\Omega$, we have the Kantorovich-Rubinstein duality  (see e.g., \cite{villani2009optimal} for more details) that gives:
\begin{align}\label{eq: KRduality}
W_1(\mu, \nu) = \sup_{f\in \cF}\Big(\int_{\Omega} f \dd\mu - \int_{\Omega} f \dd\mu\Big),
\end{align}
where $\cF$ denotes the function class of all 1-Lipschitz functions.  

\subsection{Integer Laplacian distribution}
Since our method involves private counts of data points in different regions, we will use integer Laplacian noise to ensure they are integers.
    An \emph{integer (or discrete) Laplacian distribution} \cite{inusah2006discrete} with parameter $\s$ is a discrete distribution on $\mathbb Z$ with probability density function 
    \[f(z) = \frac{1-p_\s}{1+p_\s} \exp \left( -\abs{z}/\s \right),
    \quad z \in \Z,\]
    where $p_\s = \exp(-1/\s)$.
   A random variable $Z \sim \Lap_\Z(\s)$ is mean-zero and sub-exponential with variance $\mathrm{Var(Z)\leq 2\s^2}$.








\section{Online synthetic data}
\label{sec: syn_data}


\subsection{Binary partition}\label{sec:binary_partition}
We follow the definition of binary hierarchical partition as described in \cite{he2023algorithmically}. A binary hierarchical partition of a set $\Omega$ of depth $r$ is a family of subsets $\Omega_\theta$, indexed by $\theta \in \{0,1\}^{\le r}$,  
$$
\{0,1\}^{\le k} = \{0,1\}^0 \sqcup \{0,1\}^1 \sqcup \cdots \sqcup \{0,1\}^k, 
\quad k=0,1,2\dots, 
$$
and such that $\Omega_{\theta}$ is partitioned into $\Omega_{\theta0}$ and $\Omega_{\theta1}$ for every $\theta \in \{0,1\}^{\le r-1}$. 
Here $\{0,1\}^0$ corresponds to the degenerate case where $\Omega_\emptyset:= \Omega$.
If $\theta \in \{0,1\}^j$, we refer to $j=|\theta|$ as the depth of $\Omega_\theta$. When $\Omega=[0,1]^d$ equipped with the $\ell_{\infty}$-norm, a subregion $\Omega_{\theta}$ with $\theta \in \{0,1\}^j$ has a volume of $2^{-j}$ and $\mathrm{diam}(\Omega_{\theta}) \asymp 2^{-\lfloor j/d\rfloor}$.

Let $\{\Omega_\theta,{\theta \in \{0,1\}^{\le r}}\}$ represent a binary partition of $\Omega$. Given a true data stream $(x_1,\ldots,x_t) \in \Omega^t$, the true count $n_\theta^{(t)}$ is the number of data points in the region $\Omega_\theta$ at time $t$, i.e.,
$$n_\theta^{(t)} :=\abs{\left\{ i \in [t]: \; x_i \in \Omega_\theta \right\}}.$$

We can also represent a binary hierarchical partition of $\Omega$ in a binary tree of depth $r$, where the root is labeled $\Omega$ and the $j$-th level of the tree $\mathcal{T}$ encodes the subsets $\Omega_{\theta}$ for $\theta$ at level $j$. As new data arrives, we refine the binary partition over time and update the true count $n_\theta^{(t)}$ in each subregion at time $t$.
As we continue refining the partition of $\Omega$, the binary tree $\mathcal{T}$ expands in the order of a breadth-first search, and the online synthetic data we release will depend on a noisy count $N_{\theta}^{(t)}$ of data points in each region $\Omega_{\theta}$ at time $t$.

\subsection{Main algorithm}
 We can now introduce our main algorithm for online differentially private synthetic data release for the case, described formally in Algorithm~\ref{alg: online_final}.
 

Algorithm~\ref{alg: online_final} uses the dynamic partition of the domain to generate synthetic data dynamically by adding dependent noise to perturb the counts of true data in each subregion. More precisely, it consists of the following steps:
\begin{enumerate}
    \item Refine a binary partition of $\Omega=[0,1]^d$ as time $t$ grows. Equivalently, the tree $\mathcal T$ encoding the binary partition grows over time in the breath-first search order. We will refine the partition and create all sub-regions $\Omega_{\theta}$ for all $|\theta| = j$ at timestamp $t_j=\lceil 2^j/\e\rceil$. We say $t$ is of level $j$ if $t_j\leq t<t_{j+1}$, or equivalently $\mathcal T$ has depth $j$. 
    
    Note that any sub-region $\Omega_\theta$ with $|\theta|=j$ in Algorithm~\ref{alg: online_final} only exists from level $j$: before level $|\theta|$ it was not created and after time $t_{j+1}$ it still exists and will be refined.
    
    \item For each existing region $\Omega_{\theta}$ at time $t$, to count the number of data in $\Omega_\theta$ privately, we output a perturbed count $N_{\theta}^{(t)}$ using a new \textit{Inhomogeneous Sparse Counting Algorithm} described in Algorithm~\ref{alg: inhomogeneous_sparse_couting}. For each subregion $\Omega_\theta$, an online counting subroutine $\cA_\theta$ starts as soon as $\Omega_\theta$ is created, and it outputs a noisy count for every time $t$ afterward. Privacy and accuracy guarantees for Algorithm~\ref{alg: inhomogeneous_sparse_couting} are given in \add{Part~\ref{sec: onlinecounting} of Section~\ref{sec: syn_data}}. 

    The choice of $r_0=|\theta|$ for $d\geq 2$ indicates that $\Omega_\theta$ will not load the historical data which came before level $j$ (see Algorithm~\ref{alg: inhomogeneous_sparse_couting}). The exact choices of the privacy parameters $\{\varepsilon_{j,r}\}$ in Algorithm~\ref{alg: online_final} to implement the subroutines $\mathcal A_{\theta}$ when $d\geq 2$ are given in \eqref{eq: optimal_eps}. The choices of $\varepsilon_{j,r}$ when $d=1$ are given in \eqref{eq: optimal_eps_d=1}. 
    
    \item The noisy counts in the Perturbation step could be negative and inconsistent. We post-process them to ensure they are non-negative, and the counts of subregions always add up to the count of the region in the upper level. The details of this step are given in Algorithm~\ref{alg: consistency}, see Section~\ref{sec: consistency}. 
    \item Finally, we turn the online synthetic counts in each region into online synthetic data by choosing the same amount of data points in each region whose location is independent of the true data.
\end{enumerate}

\begin{algorithm}
	\caption{Online synthetic data}
        \label{alg: online_final}
	\begin{algorithmic}
        \State {\bf Input:} Privacy budget $\varepsilon$ and an infinite data sequence $\{x_i\}_{i=1}^{\infty}$ where $x_i\in [0,1]^d$.  For each time $t$,  data points $(x_1,\dots, x_t)$ are available. 
	\State {\bf Initialization:} Set $t=t_0= 1$, $\Omega_\emptyset=\Omega$, and the depth of partition tree $r\gets 0$.
 
        \State \While{$t\in \mathbb N$}
            \While{$t\geq 2^r/\e$}
                \State $r\gets r+1$.
                \State {\bf (New binary partition)}  Partition $\Omega_\theta$ into $\Omega_{\theta 0}$ and $\Omega_{\theta 1}$ for every $|\theta|=r-1$.

                \State {\bf (New subroutines)} For every newly created $\Omega_\theta$, Initiate  a subroutine denoted by $\cA_\theta$. Here $\cA_\theta$ implements Algorithm~\ref{alg: inhomogeneous_sparse_couting} with input parameters given by starting level
                \[r_0 = \begin{cases}
                    |\theta| &\text{if } d\geq 2,\\
                    0 &\text{if }d=1
                \end{cases}\]
                and  privacy parameters $\e_{|\theta|,r_0}, \e_{|\theta|, r_0+1},\dots$.  The Input Boolean data sequence of $\cA_{\theta}$ will be specified in the Perturbation step below.
            \EndWhile 
            
    	\State {\bf (Perturbation)} For every $\Omega_\theta$, where $1\leq |\theta|\leq r$, compute the noisy online count $N_\theta^{(t)}$ using the subroutine $\cA_\theta$  with an  updated input Boolean data stream  $\left\{\mathbf{1}_{\{x_i\in \Omega_\theta\}}\right\}_{i=2^{|\theta|}}^t$.  Namely, $\mathbf{1}_{\{x_t\in \Omega_\theta\}}$ is added to the data stream.
     
    	\State {\bf (Consistency)} Transform the perturbed counts of each subregion, $\{N_\theta^{(t)}\}_{|\theta|\leq r}$, into  non-negative consistent counts  $\{\widehat N_\theta^{(t)}\}_{|\theta|\leq r}$ using Algorithm~\ref{alg: consistency}.
      
    	\State {\bf (Output)}  Output the synthetic data $\cY_t$ by choosing the locations of $\widehat N_\theta^{(t)}$ many data points arbitrarily and independently of the true data within each subregion $\Omega_\theta$ where $|\theta|=r$. 
     
            \State Let $t\gets t+1$.
        \EndWhile
	\end{algorithmic}
\end{algorithm}




\subsection{Sparse counting algorithm from \cite{dwork2015pure}}
The new Inhomogeneous Sparse Counting Algorithm (Algorithm~\ref{alg: inhomogeneous_sparse_couting}) is a generalization of the Sparse Counting Algorithm introduced in \cite{dwork2015pure}.  In this section, we provide some descriptions and properties of the algorithm from \cite{dwork2015pure} with some modifications for our applications in \add{Part~\ref{sec: onlinecounting} of Section~\ref{sec: syn_data}}.

The sparse counting algorithm with finite time horizon $T$ in \cite{dwork2015pure} is $\e$-DP with an optimal accuracy error $O(\log T)$ when the data stream is sparse. The idea of the sparse counting algorithm is to partition the timeline into multiple segments, and each segment only contains a small amount of non-zero data.
Algorithm~\ref{alg: finite_sparse_couting} we present here is a slight modification of the algorithm in \cite{dwork2015pure} by choosing a different partition threshold $T_0$, while
 the value of the partition threshold $T_0$ in \cite{dwork2015pure} is related to an extra parameter, the confidence probability. We define some terms used in the description of Algorithm~\ref{alg: finite_sparse_couting}:
 \begin{itemize}
     \item \textit{Segment}: a segment is a time interval. Algorithm~\ref{alg: finite_sparse_couting} partition the time interval $[0,T]$ into several sub-intervals called segments.
     \item \textit{Online counting subroutine}: This subroutine takes a non-negative integer data stream and outputs a private running sum at each time $t$.
 \end{itemize}

\begin{algorithm}
    \caption{Sparse counting with a finite time horizon, modified from \cite{dwork2015pure}.}
    \label{alg: finite_sparse_couting}
    \begin{algorithmic}
    
    \State {\bf Input:} Time horizon $T< \infty$, Boolean data sequence $\{X_i\}_{i=1}^T$.
    Privacy parameter $\e$. 
        
    \State {\bf Initialization:} Set $T_0 = 9\log T/\e$ to be the partition threshold and $j=1$ denoting the number of segments. $t=0$. Let $\widetilde{N}=0$ denote the private count in the previous $(j-1)$ segments.

    \State {\bf Online counting subroutine:} Start an online counting subroutine $\cA_{\text{sub}}$ with input privacy parameter $\e/2$ and input data stream to be determined later.

    \State \While{$t\leq T$}
        
        \State Set the segment count ${N_j} = 0$ and the private threshold $\widetilde T_j = T_0 + \lambda_j$, where $\lambda_j\sim \Lap_{\Z}(2/\e)$.
        
        \State \While {$t\leq T$ and $N_j + \l'_t \leq \widetilde{T_j}$, where $\l'_t\sim\Lap_\Z(2/\e)$}
             \State Set $t\gets t+1$, $N_j \gets N_j + X_t$.
             
             \State {\bf Output.} Output the same $\widetilde{N}$ for all $t$ in this segment.
        \EndWhile

        \State End the segment $j$ and set $j\gets j+1$.
        
        \State Run $\cA_{\text{sub}}$ with an  updated input Boolean data stream  $\{N_i\}_{i=1}^{j-1}$.  Namely, $N_{j-1}$ is added to the data stream.
        Update $\widetilde N$ to be the latest output of $\cA_{\text{sub}}$.
        
    \EndWhile
    \end{algorithmic}
\end{algorithm}

Algorithm~\ref{alg: finite_sparse_couting} has various choices of the online counting subroutine $\cA_{\text{sub}}$. One choice is the \emph{Binary Mechanism}  proposed in \cite{chan2011private,dwork2010differential}, which ensures differential privacy for any input data stream. The original version of \cite[Algorithm 2]{chan2011private} requires the data sequence to be Boolean, but it can also be used for non-negative integer data streams; see, for example, \cite{dwork2015pure}.

The privacy and utility guarantee of the Binary Mechanism is given in the next lemma. \add{The proof of Lemma~\ref{lem:chan_accuracy} is deferred to Appendix~\ref{sec:addtional_proofs}}.

\begin{lemma}\label{lem:chan_accuracy}
    The \emph{Binary Mechanism} is $\e$-differentially private for a finite time horizon $T$. And for any time $t\in [0,T]$, let $\mathrm{error}_t$ denote the error at time $t$ between the true count $\sum_{i=1}^t X_i$ and the output at time $t$, we have 
    \begin{equation}
\label{eq: chan_accuracy}
    \E \mathrm{error}_t \leq  \frac{C}{\e}\log^{1.5}T.
\end{equation}
\end{lemma}

With the guarantee of the subroutine using the Binary Mechanism, \cite{dwork2015pure} shows that their sparse counting Algorithm indeed improves the counting error. We will prove a similar expectation bound with the new partition threshold $T_0$ in Algorithm~\ref{alg: finite_sparse_couting}. \add{The proof of Lemma~\ref{lem: sparse} is also deferred to Appendix~\ref{sec:addtional_proofs}}.

\begin{lemma}
    \label{lem: sparse}
When choosing the online counting subroutine as the Binary Mechanism,     Algorithm~\ref{alg: finite_sparse_couting} attains $\e$-differential privacy. 
    Let $\mathrm{error}_t$ denote the counting error between true count $\sum_{i=1}^t X_i$ and the output at time $t$. Then, for any fixed $t$, there is an accuracy bound
    \[\E \mathrm{error}_t \lesssim (\log T + \log^{1.5} n)/\e,\]
    where $T$ is the bound for the time horizon, and $n$ is the number of the non-zero elements in the input data stream.
\end{lemma}

\subsection{Noisy online count}\label{sec: onlinecounting}


We now provide a detailed description of the Perturbation step in Algorithm~\ref{alg: online_final}. Our goal is to output a noisy count of the points $x_1,\dots, x_t$ in $\Omega_{\theta}$, denoted by $N_{\theta}^{(t)}$. Since \[ n_{\theta}^{(t)}=\sum_{i=1}^t \mathbf{1}_{\{ x_i\in \Omega_{\theta}\}},\]
this step is closely related to the differentially private count release under continual observation for  Boolean data studied in \cite{dwork2010differential,chan2011private,dwork2015pure}. 

\begin{algorithm}
    \caption{Inhomogeneous sparse counting}
    \label{alg: inhomogeneous_sparse_couting}
    \begin{algorithmic}
    \State {\bf Input:} Output starting level $r_0$. Boolean data sequence $\{X_t\}_{t=2^{r_0}}^\infty$. Noise parameters $\e_{r_0},\e_{r_0+1},\dots$.
        
    \State {\bf Initialization:} Set the finite private count  $\widetilde{S}\gets 0$, the current level $r\gets r_0$, and the timestamps $t_{r_0}=\lceil2^{r_0}/\e\rceil$, $t_{r_0+1}=\lceil 2^{r_0+1}/\e \rceil$. 
    
    \State \While {$r_0 \leq r < \infty$}
        \State {\bf Counting subroutine:} For $t\in [t_r, t_{r+1})$, apply Algorithm~\ref{alg: finite_sparse_couting} with time horizon $t_{r+1 }- t_{r}$ and privacy parameter $\e_r/2$. Record the outputs $c_{t_r},\dots,c_{t_{r+1}-1}$.
        
        \State {\bf Output.} Output $\widetilde{S} + c_t$ as the private count  at time $t$ for each $t\in [t_r, t_{r+1})$.
        
        \State Update \[\widetilde{S} \gets \widetilde{S} + \sum_{t=t_r}^{t_{r+1}-1} X_t + \Lap_{\Z}(2/\e_r)\] and start a new level with $r\gets r+1$, $t_{r+1}\gets \lceil 2^{r+1}/\e \rceil$.
    \EndWhile
    \end{algorithmic}
\end{algorithm}

Algorithm~\ref{alg: inhomogeneous_sparse_couting} is based on Algorithm~\ref{alg: finite_sparse_couting} and uses integer Laplacian noise with different variances in different time intervals. 
We now give several definitions to describe Algorithm~\ref{alg: inhomogeneous_sparse_couting}:
\begin{itemize}
    \item \textit{Time level}: Starting from level $0$, we say time $t$ is at \textit{level} $j$ if $2^r/\e \leq t< 2^{r+1}/\e$. In Algorithm~\ref{alg: inhomogeneous_sparse_couting}, we process the data stream level by level, where level $r$ starts from the timestamp $t_r = \lceil 2^r/\e \rceil $. 
    \item \textit{Starting level}: We set an additional input $r_0$ to indicate the level from which the output starts. More precisely, the output of Algorithm~\ref{alg: inhomogeneous_sparse_couting} start from time $t_{r_0} = \lceil 2^{r_0}/\e \rceil$. 
     We use $\widetilde S$ to store the private count from starting level $r_0$ to level $(r-1)$, and $\widetilde S$ does not include the count of data points arriving before the starting level $r_0$.
     
    \item \textit{Counting subroutine}: The subroutine Algorithm~\ref{alg: finite_sparse_couting} is an online counting algorithm with finite time horizon. It takes a Boolean data series as input and outputs the private counts of the first $t$ data points at any time $t$. Here, we apply Algorithm~\ref{alg: finite_sparse_couting} to the Counting subroutine to obtain noisy counts $c_t$ of the number of 1's arriving in the time interval $[t_r,t]$  for each $t\in [t_r, t_{r+1})$. 
    
    \item \textit{Update of $\widetilde S$}: During the counting subroutine, $\widetilde S$ is not updated. It is only updated at the end of the time level $r$ by adding a noisy count $\sum_{t=t_r}^{t_{r+1}-1} X_t + \Lap_{\Z}(2/\e_r)$.
\end{itemize}



The following lemma is a privacy guarantee for Algorithm~\ref{alg: inhomogeneous_sparse_couting}.
We show Algorithm~\ref{alg: inhomogeneous_sparse_couting} gives differential privacy under different notions of neighboring data sets $\cX, \cX'$ depending on when their different data points arrive in the data stream. 
\begin{lemma}
    \label{lem: inhomogeneous}
    Let $\cA$ be Algorithm~\ref{alg: inhomogeneous_sparse_couting}. For two datasets $\cX,\cX'$ which differ on one data point at time $t\in [t_r,t_{r+1})$, and for any measurable subset $S$ in the range of $\cA$, the following holds:
    \begin{enumerate}
        \item If $r\geq r_0$, $\Pr{\cA(\cX)\in S} \leq e^{\e_r} \cdot \Pr{\cA({\cX'})\in S}$.
        \item If $r< r_0$, $\Pr{\cA(\cX)\in S} = \Pr{\cA({\cX'})\in S}$.
    \end{enumerate}
\end{lemma}
\begin{proof}
For such $\cX,\cX'$ in the theorem, when $r< r_0$, one can notice that $X_t$ does not appear in the algorithm. Therefore, the second assertion holds.


When $r \geq r_0$, to have the different data value at time $t$ would make the following two influences:
    \begin{enumerate}
        \item When $\widetilde{S}$ first counts $X_t$ privately applying the $\e_r/2$-differentially private subroutine;
        \item When updating the count $\widetilde S$, we add noise $\Lap_{\Z}(2/\e_r)$, which implies another privacy budget $2/\e_r$.
    \end{enumerate}
    By the parallel composition property of differential privacy \cite{dwork2014algorithmic}, we know for the data at time $t$, the algorithm is $\e_r$ differentially private.
\end{proof}

Lemma~\ref{lem: inhomogeneous_accuracy} bounds the difference between a noisy count and the true count in Algorithm~\ref{alg: inhomogeneous_sparse_couting}  for different time intervals.
\begin{lemma}\label{lem: inhomogeneous_accuracy}
    For each time $t\in [t_r,t_{r+1})$, $\widetilde{S}+c_t$ is the output of the noisy count at time $t$ in Algorithm~\ref{alg: inhomogeneous_sparse_couting}. We have  
    \[ \E | \widetilde{S}+c_t-S_t|\lesssim \sum_{i=1}^{t_{r_0}-1} X_i + \sum_{i=r_0}^{r-1} \frac{1}{\e_i} + \frac{\log t + \log^{1.5} n_r}{\e_r},\]  
    where $S_t:=\sum_{i=1}^t X_i$ is the true count at time $t$ and $n_r:=\sum_{i=t_r}^{t_{r+1}}X_i$.
\end{lemma}

\begin{proof}
    The accuracy part follows from the results of the subroutine and Laplacian mechanism. 
    At time $t\in [t_r,t_{r+1})$, by the result of  Lemma~\ref{lem: sparse}, we know the error of count $c_t$ at time $t$ has bound 
    \[\E \bigg|c_t - \sum_{i=t_r}^t X_i\bigg| \lesssim \frac{\log t + \log^{1.5} n_r}{\e_r}.\]
    And by the Laplacian mechanism, we know the count $\widetilde S$ has the accumulating error from each level (starting from $r_0$), namely
    \[\E \bigg|\widetilde S - \sum_{i=t_{r_0}}^{t_{r} -1} X_i\bigg| \lesssim \sum_{i=r_0}^{r-1} \frac{1}{\e_i}.\]
    Therefore, considering that $\widetilde S$ ignored the the data before level $r_0$, we deduce that
    \begin{align*}
        \E \abs{N_t-S_t} & = \E \bigg|{\widetilde S + c_t - \sum_{i=1}^t X_i}\bigg| \\
        & \leq \sum_{i=1}^{t_{r_0}-1} \! X_i + \E \bigg|\widetilde S - \!\sum_{i=t_{r_0}}^{t_{r} -1}\! X_i\bigg| + \E \bigg|c_t - \!\sum_{i=t_r}^t \!X_i\bigg| \\
        & \lesssim \sum_{i=1}^{t_{r_0}-1} X_i + \sum_{i=r
        _0}^{r-1} \frac{1}{\e_i} + \frac{\log t + \log^{1.5} n_r}{\e_r}.
    \end{align*}
\end{proof}

\subsection{Consistency}\label{sec: consistency}

The consistency step is adapted from \cite[Algorithm 3]{he2023algorithmically}, as described in Algorithm~\ref{alg: consistency}. We first convert all negative noisy counts to $0$ and then transform the entire sequence of noisy counts for the hierarchical partition into a consistent count, ensuring the counts from subregions add up to the count at the next level. In terms of the partition tree $\mathcal{T}$ described in Section~\ref{sec:binary_partition}, consistency means that the counts from any two nodes sharing the same parent will sum to the count from the parent node. This step is crucial for Algorithm~\ref{alg: online_final} to obtain a probability measure close to the empirical measure $\mu_{\cX_t}$ in the Wasserstein distance.

\begin{algorithm}	
\caption{Consistency} \label{alg: consistency}
\begin{algorithmic} 

\State {\bf Input:} Integer sequence $(n'_\theta)_{\theta \in \{0,1\}^{\le r}}$  corresponding to the private count of each $\Omega_{\theta}$.
\State For the case $j=0$, set $m \gets \max(n', 0)$. 

\For{$j=0,\ldots,r-1$}
    \For{$\theta \in \{0,1\}^j$}
        \State Set the counts 
        \[n'_{\theta0} \gets \max(n'_{\theta0}, 0), n'_{\theta1} \gets \max(n'_{\theta1}, 0).\]

        \State Transform the vector $(n'_{\theta0},n'_{\theta1}) \in \Z_+^2$ into any vector $(m_{\theta0},m_{\theta1}) \in \Z_+^2$ s.t.
        \[m_{\theta0} + m_{\theta1} = m_\theta;\quad (m_{\theta0}-n_{\theta0})(m_{\theta1}-n_{\theta1})\geq 0\]
    \EndFor
\EndFor

\State {\bf Output:} non-negative integers $(m_\theta)_{\theta \in \{0,1\}^{\le r}}$.

\end{algorithmic}
\end{algorithm}

\section{Proof of Theorem \ref{thm: main} when $d\geq 2$}\label{sec:proof}
We now prove that when $d\geq 2$, Algorithm~\ref{alg: online_final} satisfies the privacy and accuracy guarantee in Theorem~\ref{thm: main}.     To complete our proof, in Algorithms~\ref{alg: online_final} and ~\ref{alg: inhomogeneous_sparse_couting}, we choose privacy parameters
\begin{equation}
\label{eq: optimal_eps}
    \e_{j,r}  = C_1\e 2^{(j-r)(1-1/d)/2}, \; \text{where }C_1=\frac{1-2^{-(1-1/d)/2}}2.
\end{equation}
Denote $\alpha\coloneq 2^{(1-1/d)/2}\in [2^{1/4}, \sqrt{2})$ and we can check
\begin{align}\label{eq:sumeps_d}
\sum_{j=1}^s \e_{j,s} = \sum_{j=1}^s C_1\e \alpha^{j-s} = \frac{C_1\e (1 - \alpha^{-s})}{1-\alpha^{-1}}\leq \frac{C_1 \e}{1-\alpha^{-1}}\leq \frac{\e}{2}.
\end{align}

\subsection{Privacy}


\begin{proposition} 
\label{prop: online_syndata_privacy}
For any choice of privacy parameters $\varepsilon_{j,r}$ satisfying 
\begin{align}
     \sup_{s\geq 0} \sum_{j=1}^s \e_{j,s}\leq \frac{\e}{2}, \notag
\end{align} we have Algorithm \ref{alg: online_final} is $\e$-differentially private. In particular, with the choice of parameters in  Equation~\eqref{eq: optimal_eps}, Algorithm \ref{alg: online_final} is $\e$-differentially private.
\end{proposition}


\begin{proof}
Since the privacy budget in Algorithm~\ref{alg: online_final} is only spent on the Perturbation step, we only need to show this step is $\e$-differentially private.

    Consider two neighboring data sets $\cX,\cX'$, which are the same except for $x_t\in \cX, x_t'\in \cX'$ arriving at time $t$. Suppose the partition $\mathcal T$ at time $t$ has depth $r=\lfloor \log_2 (\e t) \rfloor$. Then, the true count of $\Omega_\theta$ corresponding to $\cX$ and $\cX'$ are the same except for at most two subregions at each level in $\mathcal T$, and they form two paths of length $r$ in the tree.  For $x_t$, let us denote these subregions \[x_t\in\Omega_{\theta_r} \subset \cdots \subset \Omega_{\theta_1}\subset \Omega.\]  On the other hand, once $x_t$ is given, we know exactly the corresponding subregions in $\mathcal T$ at level $r+1,r+2,\dots$ will contain $x_t$ in the future as soon as they are created. This gives us an infinite sequence of subregions \[\Omega_{\theta_r} \supset \Omega_{\theta_{r+1}}\supset \cdots.\] Similarly we can also obtain an infinite sequence $\Omega\supset\Omega_{\theta_1'} \supset \Omega_{\theta_{2}'}\supset \cdots$ containing $x_t'$.
    
    Consider the first sequence. As the difference of $\cX,\cX'$ at time $t$ will only influence the counts in $\Omega_{\theta_j}$, $j\geq 0$. We consider the subroutine $\cA_{\theta_j}$ in each of the regions $\Omega_{\theta_j}$ for all $j\geq 0$. There are two cases: 
    \begin{enumerate}
        \item When $0<j\leq r$, $\Omega_{\theta_j}$ counts the data $x_t$ at time $t$. So by Lemma~\ref{lem: inhomogeneous}, we protect the privacy of $x_t$ with parameter $\e_{j,r}$.
        
                
        \item When  $j>r$, by the Initialization step in Algorithm~\ref{alg: inhomogeneous_sparse_couting}, $\cA_{\theta_j}$ and the private counts $N_{\theta_j}^{(t)}$ no longer depend on the value of $X_t$. 
    \end{enumerate}
    By the parallel composition rule of differential privacy \cite[Theorem 3.16]{dwork2014algorithmic} and taking a supremum over all possible $t$, we have Algorithm~\ref{alg: online_final} is differentially private with parameter 
    \[\sup_{s\geq 0} \sum_{j=1}^s \e_{j,s}\leq \frac{\e}{2},\]
    where the inequality above holds due to \eqref{eq:sumeps_d}.
    
    The same argument holds for the second subregion sequence containing $x_t'$. Hence the whole algorithm is $\e$-differentially private by applying the parallel composition rule again.
\end{proof}

\subsection{Accuracy}

    Now we consider the accuracy of the output in Wasserstein distance at time $t$ with the corresponding level $r=\lfloor\log_2 (\e t)\rfloor$.     To ensure the accuracy of the consistency step, we include the following lemma from \cite{he2023algorithmically}.
    \begin{lemma}[\cite{he2023algorithmically}, Theorem 4.2]\label{lem:thm42}
        With data set $\cX$ of size $n$ and the binary partition structure $(\Omega_\theta)_{|\theta|\leq r}$ of depth $r$, let $\lambda_\theta$ denote the noise adding to the true count in $\Omega_\theta$. Then by forcing consistency and generating synthetic data $\cY$ of size $n$ according to the consistent private counts, the Wasserstein error is
        \begin{align*}
        &\E W_1(\mu_{\cX}, \mu_\cY)\\
        \lesssim & \frac{1}{n}\sum_{j=0}^{r-1}\sum_{|\theta| = j} \E\big[\max\{|\lambda_{\theta 0}|,|\lambda_{\theta 1}|\}\big]\diam(\Omega_\theta) +\delta.
        \end{align*}
        Here $\delta =\max_{|\theta|=r}\diam(\Omega_\theta) $ is the maximal diameter of the subregions of depth $r$.
    \end{lemma}

We also need the following estimates in the proof.
    \begin{lemma}
\label{lem: shifting_sum}
    Suppose $\alpha>1$ is a constant and $r=\lfloor\log_2 (\e t)\rfloor$, then
    \[S'= \sum_{j=1}^r \alpha^j (r-j+1) \lesssim \alpha^r,\]
    \[S = \sum_{j=1}^r \alpha^j (r-j+1)^2 \lesssim \alpha^r.\]
\end{lemma}
\begin{proof}
We have the following holds:
\begin{align}
    \alpha S' &= \sum_{j=1}^r \alpha^{j+1} (r-j+1) = \sum_{j=2}^{r+1} \alpha^j (r-j+2), \notag  \\
    (\alpha - 1)S' &= \alpha S - S =\sum_{j=2}^{r+1}\alpha^{j} + \alpha^{r+1} - \alpha r, \notag\\
    \Longrightarrow S' &= \frac{\alpha^{r+1} - \alpha^2}{\alpha-1} - \alpha r \lesssim \alpha^r. \label{eq:upperboundS}
\end{align}
For the sum $S$, again, there is
\begin{align*}
    \alpha S &= \sum_{j=1}^r \alpha^{j+1} (r-j+1)^2 = \sum_{j=2}^{r+1} \alpha^j (r-j+2)^2, \\
    (\alpha - 1)S &= \alpha S - S =\sum_{j=2}^{r+1}\alpha^{j}(2r-2j+3) + \alpha^{r+1} - \alpha r^2.
\end{align*}
Applying \eqref{eq:upperboundS}, we have 
\begin{align*}
    S &= \frac{1}{\alpha-1} \left(\sum_{j=2}^{r+1}\alpha^{j}(2r-2j+3) + \alpha^{r+1} + \alpha r^2\right) \\
    &\lesssim \alpha^r + \alpha^{r+1} - \alpha r^2 \lesssim \alpha ^r
\end{align*}
    as desired.
\end{proof}


Next we are ready to prove the accuracy bound \eqref{eq:accuracy_finite}  for $d\geq 2$.   
\begin{proof}[Proof of \eqref{eq:accuracy_finite} for $d\geq 2$]
    Let $\l_{\theta}^{(t)}=N_{\theta}^{(t)}-n_{\theta}^{(t)}$ be the counting noise of subregion $\Omega_\theta$ at time $t$ with $1\leq j=|\theta|\leq r$, where $t_r\leq t< t_{r+1}$. Recall that $t_i$, defined as $t_i=\lceil2^i / \e\rceil$, denotes the timestamp when level $i$ starts. Note that $\Omega_\theta$ is created at time level $j$, which is also the starting level parameter in subroutine $\cA_{\theta}$. 
    
    By Lemma~\ref{lem: inhomogeneous_accuracy}, we have the upper bound of the noise at time $t\geq \frac{e}{\e}$ for $\Omega_\theta$,
    \begin{align}\label{eq:lambda_theta_t}
    \E \abs{\l_{\theta}^{(t)}} \lesssim \;&\Big|{\{x_s\mid s< t_{j}, x_s\in \Omega_{\theta}\}}\Big| + \nonumber \\
    &\quad \sum_{i=j}^{r-1} \frac{1}{\e_{j,i}} + \frac{\log t + \log^{1.5}{n_j}}{\e_{j,r}}.
    \end{align}
    Applying Lemma~\ref{lem:thm42}, at a fixed time $t\geq \frac{e}{\e}$, we have
    
    \begin{align}
        \E &W_1(\mu_{\cX_t}, \mu_{\cY_t})  
        \leq \nonumber \\
        &\frac{1}{t} \left[
    	 \sum_{j=0}^{r-1} \sum_{\theta \in \{0,1\}^j} \E \left[ \max \left\{ \,\abs{\lambda_{\theta0}^{(t)}}, \abs{\lambda_{\theta1}^{(t)}} \right\} \right]\diam(\Omega_\theta) \right]\! + \delta, \label{eq:he2023}
    \end{align}
    where $\delta = \max_{|\theta|=r}\diam(\Omega_\theta)$ denotes the maximal diameter of the subregions of depth $r$. For $\Omega=[0,1]^d$, we have $\diam(\Omega_{\theta}) \asymp 2^{-|\theta|/d}$ and there are $2^{|\theta|}$ many different subregions of such size. 
    
    Note that for fixed $j$, $\e_{j,i}$ decreases as $i$ increases.
    From \eqref{eq:he2023} and \eqref{eq:lambda_theta_t}, we have for $t\geq \frac{e}{\e}$,
    \begin{align}
    \label{eq: pmm_cite}
        &\E W_1(\mu_{\cX_t}, \mu_{\cY_t}) \nonumber\\
        \leq & \frac{1}{t} \sum_{j=0}^{r}\sum_{|\theta|=j} \E \abs{\lambda_{\theta}^{(t)}} \cdot 2^{-j/d} + 2^{-r/d} \nonumber\\
        \lesssim & \frac{1}{t} \sum_{j=0}^{r} \sum_{|\theta|=j} \bigg( \Big|{\{x_s\mid s< t_j, x_s\in \Omega_{\theta}\}}\Big| + \nonumber\\
        & \quad \sum_{i=j}^{r-1} \frac{1}{\e_{j,i}}
        + \frac{\log t + \log^{1.5}{(n_\theta^{(t)}+1)}}{\e_{j,r}} \bigg)\cdot 2^{-j/d} + 2^{-r/d} \nonumber \\
        \lesssim &\frac{1}{t} \sum_{j=0}^{r} \bigg(\frac{2^j}{\e} \! + \!\!\! \sum_{|\theta|=j}\!\!\frac{\log t + \log^{1.5}{(n_\theta^{(t)}+1)}}{\e_{j,r}} \bigg)\!\cdot 2^{-j/d} \!+\! 2^{-r/d} 
    \end{align}
    
Since
    \[\left(\log^{1.5} x\right)'' = \frac{1.5\big(\frac{1}{2\sqrt{\log x}}-\sqrt{\log x}\big)}{x^2} \leq 0, \quad\text{ if } x\geq \sqrt{e},\]
    the function $\log^{1.5} (x+2)$ is  concave  on $[0,+\infty)$. Therefore, for fixed $j$, we can apply Jensen's inequality when summing the $\log^{1.5} (n_\theta^{(t)}+1)$ terms over all $|\theta|=j$ and obtain
    \begin{align*}
        \sum_{|\theta|=j} \log^{1.5}(n_\theta^{(t)} +1) & \leq \sum_{|\theta|=j} \log^{1.5}(n_\theta^{(t)} + 2)\\
        & \leq 2^j \log^{1.5}\bigg(2^{-j}\sum_{|\theta|=j}n_\theta^{(t)} + 2\bigg) \\
        & \leq 2^j \log^{1.5}\bigg(\frac{t}{2^j} + 2\bigg).
    \end{align*}
    Substitute the result above into \eqref{eq: pmm_cite} and we have
    \begin{align}\label{eq:expectation_accuracy_bound}
        &\E W_1(\mu_{\cX_t}, \mu_{\cY_t}) \nonumber\\
        &\lesssim \frac{2^{r(1-1/d)}}{\e t} + 2^{-r/d} + \nonumber\\
        &\quad\quad \quad \frac{1}{t}\sum_{j=0}^{r} \frac{1}{\e_{j,r}} \Big(\log t + \log^{1.5}\Big(\frac{t}{2^j} + 2\Big)\Big) 2^{j(1-1/d)}
    \end{align}

    As $t\in [t_r,t_{r+1})$, we have
    \begin{align*}
        \log(t/2^j + 2) &\leq \log(2^{r-j+1}/\e+2) \\
        & = \log(2^{r-j+1}) + \log \Big(\frac{1}{\e}+ 2^{j-r}\Big)\\
        &\lesssim (r-j+1) + \log \Big(\frac{1}{\e}+ 1\Big).
    \end{align*}

    To attain $\e$-differentially privacy, we can choose the privacy parameters to optimize the accuracy in Wasserstein distance given in \eqref{eq:expectation_accuracy_bound}. One of the nearly best choices is given in \eqref{eq: optimal_eps}.
    Therefore, for the second term in \eqref{eq:expectation_accuracy_bound}, we deduce that
    \begin{align}
    \label{eq: sum_sqrt_Delta}
        &\sum_{j=1}^r\, \frac{1}{\e_{j,r}} \Big(\log t + \log^{1.5}\Big(\frac{t}{2^j} + 2\Big)\Big)2^{j(1-1/d)} \nonumber\\
        \lesssim \,& \frac{2^{r(1-1/d)/2}}{\e} \cdot \nonumber\\
        &\; \sum_{j=1}^r \Big(\log t + (r-j+1)^2 + \log^2 \Big(\frac{1}{\e}+ 1\Big)\Big) 2^{j(1-1/d)/2} \nonumber \\
        \lesssim \,& \frac{2^{r(1-1/d)}}{\e} \Big(\log t + \log^2 \Big(\frac{1}{\e}+ 1\Big)\Big).
    \end{align}
    Here, the last inequality uses Lemma~\ref{lem: shifting_sum} with $\a = 2^{\frac{1-1/d}{2}}>1$ when $d\geq 2$.
    
    Therefore, when $d\geq 2$, with \eqref{eq:expectation_accuracy_bound} and \eqref{eq: sum_sqrt_Delta} we have
    \begin{align*}
        &\E W_1(\mu_{\cX_t}, \mu_{\cY_t})\\ 
        \lesssim & \frac{2^{r(1-1/d)}}{\e t} + \frac{\log t + \log^2 \big(\frac{1}{\e}+ 1\big)}{\e t} \cdot 2^{r(1-1/d)}+ (\e t)^{-1/d} \\
        \lesssim & \Big(\log t + \log^2 \Big(\frac{1}{\e}+ 1\Big)\Big) \cdot (\e t)^{-1/d}.
    \end{align*}

 When $t\geq e/\e+\exp(\log^2(1/\e+1))$, the inequality above can be simplified as follows
  \begin{align*}
        \E W_1(\mu_{\cX_t}, \mu_{\cY_t}) &\lesssim \log (t ) (\e t)^{-1/d}.
    \end{align*}
    
This finishes the proof.
\end{proof}


\section{Proof of Theorem~\ref{thm: main} for $d=1$}\label{sec:proof_d_1}

We now prove  the privacy and accuracy guarantee of Algorithm~\ref{alg: online_final} for $d=1$ with a different choice of privacy parameters $\varepsilon_{j,r}$ given in \eqref{eq: optimal_eps_d=1}.


\begin{proposition}
    When $d=1$, Algorithm~\ref{alg: online_final} satisfies $\e$-online privacy with privacy parameters
    \begin{equation}
        \label{eq: optimal_eps_d=1}
        \e_{j,r} = \frac{3}{\pi^2} \cdot\frac{\e}{(j+1)^2}
    \end{equation}
    for all $r\geq j$.
    Moreover, for any time $t\geq e/\e$, 
    \[\mathbb EW_1(\mu_{\mathcal X_t},\mu_{\mathcal Y_t})\lesssim 
    \frac{1}{\e t}\log^{3}(\e t)\log^{1.5}t.\]
\end{proposition}

\begin{proof}
The privacy guarantee follows from the proof to Proposition~\ref{prop: online_syndata_privacy}. Since we have $r_0=0$ when $d=1$, for every data $x_t$, it influences the true count of $\Omega_\theta$ for exactly one $\theta=\theta_j$ with $|\theta|=j, j\geq 1$. Therefore, following the the proof to Proposition~\ref{prop: online_syndata_privacy}, we know the Algorithm~\ref{alg: online_final} with $d=1$ is $\e$-differentially private as 
\[2\sup_{s\geq 1}\sum_{j=1}^\infty \e_{j,s} = \frac{6}{\pi^2} \sum_{j=1}^\infty \frac{\e}{(j+1)^2} = \e.\]

The accuracy in Wasserstein distance follows from \add{the accuracy proof to Theorem~\ref{thm: main} in Section~\ref{sec:proof}}. Note that when $d=1$, we apply the counting subroutine Algorithm~\ref{alg: inhomogeneous_sparse_couting} with parameter $r_0=0$, which means we do not ignore the earlier data for any $\Omega_\theta$. Therefore, in \eqref{eq:lambda_theta_t}, there is no longer the first term, which indicates the number of data we ignore for $\Omega_\theta$. Hence, following  the proof to \eqref{eq:expectation_accuracy_bound}, at any time $t$, the Wasserstein error between the true data set $\cX_t$ and the synthetic data set $\cY_t$ at time $t$ is 
\begin{align*}
    &\E W_1(\mu_{\cX_t}, \mu_{\cY_t}) \\
        &\lesssim \frac{1}{t}\sum_{j=0}^{r} \frac{1}{\e_{j,r}} \Big(\log t + \log^{1.5}\Big(\frac{t}{2^j} + 2\Big)\Big) 2^{j(1-1/d)} + 2^{-r/d},
\end{align*} 
where $r$ is the level at time $t$, i.e., $\lceil 2^r/\e\rceil\leq t<\lceil 2^{r+1}/\e\rceil$. Substitute $d=1$ and the new choices of privacy parameters, we have 
\begin{align*}
    &\quad \,\E W_1(\mu_{\cX_t}, \mu_{\cY_t}) \\
        & \lesssim \frac{1}{t}\sum_{j=0}^{r} \frac{1}{\e_{j,r}} \Big(\log t + \log^{1.5}\Big(\frac{t}{2^j} + 2\Big)\Big) + 2^{-r} \\ 
        & \lesssim \frac{1}{t}\sum_{j=0}^{r} \frac{(j+1)^2}{\e} \Big(\log t + \log^{1.5}\Big(\frac{t}{2^j} + 2\Big)\Big)+ 2^{-r}\\
        &\lesssim \frac{1}{t}\sum_{j=0}^{r} \frac{(j+1)^2}{\e} \log^{1.5} t+ \frac{1}{\e t}\\
        &\lesssim \frac{(r+1)^{3} \log^{1.5}t}{\e t}\\
        &\lesssim \frac{(\log(\e t)+1)^3\log^{1.5}t}{\e t}.
\end{align*} 
When $t\geq e/\e$,  the inequality above becomes
\begin{align*}
  \E W_1(\mu_{\cX_t}, \mu_{\cY_t}) 
        & \lesssim \frac{\log^3(\e t)\log^{1.5}(t)}{\e t} .  
\end{align*}
This finishes the proof.
\end{proof}

\section{Time complexity}\label{sec: running_time}

We consider the running time for the algorithms to output after the input data arrives at a fixed timestamp $t$. The accumulating time complexity of the algorithms for time $1,\dots,t$ need a further sum of the time complexity at a fixed timestamp.

The Binary Mechanism in \cite{dwork2010differential, chan2011private} has time complexity $O(\log t)$ to give the output at time $t$, as it sums over $O(\log t)$ many Laplacian random variables. Same time complexity holds for sparse counting Algorithm~\ref{alg: finite_sparse_couting}, as it checks the partition threshold with $O(1)$ time and runs Binary Mechanism as a subroutine. 

As for Algorithm~\ref{alg: inhomogeneous_sparse_couting}, \emph{Inhomogenous Sparse Counting}, it is connected by multiple implementations of Algorithm~\ref{alg: finite_sparse_couting}. Furthermore, for a given $t\in [t_r, t_{r+1})$, only one such subroutine is active with time horizon $t_{r+1} - t_r \asymp 2^{r}/\e\asymp t$, so the time complexity is also $O(\log t)$.


For our main Algorithm~\ref{alg: online_final}, there is another variable $d$, the data dimension. We can decompose the procedure of Algorithm~\ref{alg: online_final} at time $t$ as the following steps:
\begin{itemize}
    \item (Partition) For each fixed time $t$ in Algorithm~\ref{alg: online_final}, there are $O(\e t)$ many sub-regions $\Omega_\theta$ in the binary partition tree of $[0,1]^d$, and when further partition happens, $O(\e t)$ many subregions are created.
    \item (Perturbation) Whenever a new data comes at timestamp $t$, it takes $O(\log (\e t))$ complexity to determine the subregions where the new data belongs. Afterwards, for all subroutines $\cA_\theta$'s, they in total take running time $O(\e t\log t)$ by the result above for Algorithm~\ref{alg: inhomogeneous_sparse_couting}.
    \item (Consistency and Output) The time complexity is $O(\e t)$ for the consistency step and $O(dt)$ for the output, as the output is $d$-dimensional data set of size $t$.
\end{itemize}
Therefore, the whole Algorithm~\ref{alg: online_final} has time complexity $O(dt+\e t\log t)$ to output at time $t$.



\section*{Acknowledgement}
{
R.V. is partially supported by NSF grant DMS-1954233, NSF grant DMS-2027299, U.S. Army grant 76649-CS, and
NSF-Simons Research Collaborations on the Mathematical and Scientific Foundations of Deep Learning. Y.Z. is partially supported by NSF-Simons Research Collaborations on the Mathematical and Scientific Foundations of Deep Learning and the AMS-Simons Travel Grant.}


\bibliographystyle{plain}
\bibliography{ref}

\appendix


\section{Additional proofs for online counting algorithms} 
\label{sec:addtional_proofs}

\begin{proof}[Proof of Lemma~\ref{lem:chan_accuracy}]
Lemma~\ref{lem:chan_accuracy} is a modified version of the following lemma, with the utility bound in expectation.
\begin{lemma}[Corollary 4.8 in \cite{chan2011private}]
\label{lem: chan_hybrid}
    The \emph{Binary Mechanism} is $\e$-differentially private for an infinite time horizon. And for any time $t\in [0,T]$ and $\beta>0$, with probability at least $1-\beta$, the counting error at time $t$ is bounded by $\frac{C}{\e}\cdot \log^{1.5}T \cdot\log \frac{1}{\b}$.
\end{lemma}

Although such error bound is in probability, we can easily transform it into a similar expectation bound. In fact, let $\mathrm{error}_t$ denote the error at time $t$ between the true count $\sum_{i=1}^t X_i$ and the output at time $t$, we have
\[ \E \mathrm{error}_t = \int_{0}^\infty \Pr{\mathrm{error}_t > u} \dd u.\]
After  a change of variable $u = \frac{C}{\e}\cdot \log^{1.5}T \cdot\log \frac{1}{\b}$ or $\beta = \exp \big(-\frac{\e u}{C\log^{1.5} T} \big)$,
we can compute
\begin{align*}
    &\quad \,\E \mathrm{error}_t \\
    & = \int_{0}^{\infty} \Pr{\mathrm{error}_t > u}\dd u\\
    & = \int_{0}^1 \Pr{\mathrm{error}_t > \frac{C}{\e}\cdot \log^{1.5}t \cdot\log \frac{1}{\b}} \frac{C \log^{1.5} T }{\beta\e} \dd \beta \\
    & \leq  \int_{0}^1 \beta \cdot \frac{C \log^{1.5} T }{\beta\e} \dd \beta\\
    & = \frac{C}{\e}\log^{1.5} T. 
\end{align*}
Therefore, we have the expectation error bound
\begin{equation*}
    \E \mathrm{error}_t \leq  \frac{C}{\e}\log^{1.5} T.
\end{equation*}
\end{proof}

\begin{proof}[Proof of Lemma~\ref{lem: sparse}]
    The privacy part follows from the original proof in \cite[Theorem 3.1]{dwork2015pure}.
    We focus on the accuracy guarantee.

    The algorithm gives a private partition of the time interval and then treats each segment in the partition as a timestamp in the online counting subroutine. We will first prove that there are at most $n+1$ many segments in the partition. 
    
    Note that there are $2T$ many independent $\Lap(2/\e)$ random variables in total. Therefore, by a simple union bound argument, with probability $1/T$, their magnitudes are uniformly bounded by $B = \frac{2}{\e}(2\log T + \log 2)$. Conditioning on this event, we know $S_j+\Lap_\Z(2/\e) > \widetilde{T_j}$ implies that $\abs{S_j-T_0}\leq 2B$ and $S_j > T_0 - 2B > 0$. So whenever a segment is sealed, its true count is non-zero; hence, we have at most $n+1$ segments (in case the last one has not been sealed). 

    Now, we can compute the expectation of error. For the case where $2T$ many $\Lap(2/\e)$ random variables share the uniform bound $B$, the counting error consists of two parts: (1) the error from the online counting subroutine $\cA_{\text{sub}}$ and (2) the approximation error when ignoring the counts within time $[t_j, t]$ (i.e. $S_j$ in the algorithm). The first part is bounded in \eqref{eq: chan_accuracy}, and the second error is bounded by $T_0 + 2B$ (as $S_j+\Lap_\Z(2/\e) \leq \widetilde{T_j}$). So the total error is $O(\frac{1}{\e}(\log T + \log ^{1.5} (n+1)))$. More precisely, the discussion can be written as the following inequality, where $c_t$ and $c_{t_j}$ denote corresponding true counts at time $t, t_j$:
    \begin{align*}
        \abs{c_t-\widetilde S} &\leq \abs{c_{t_j} - \widetilde S} + \abs{c_t - c_{t_j}}\\
        &= \abs{c_{t_j} - \widetilde S} + S_j \\
        &\leq \abs{c_{t_j} - \widetilde S} + T_0 + \abs{S_j - T_0}\\
        &\lesssim \frac{\log^{1.5} (n+1)}{\e} + \frac{\log T}{\e}.
    \end{align*}
    
    For the other case, if the uniform upper bound fails with probability $1/T$, as the content of each segment is no longer available, we only have a trivial upper bound $T$ for the number of segments. So the first error term from $\cA_{\text{sub}}$ becomes $O(\frac{1}{\e}\log^{1.5} T)$. And for the second part, we have a trivial upper bound $|S_j|\leq n\leq T$. Therefore, we have error $O(\frac{1}{\e}\log^{1.5} T + n)$.
 
    By the law of total expectation, we have
    \begin{align*}
        \E \mathrm{error}_t &\lesssim 
    \Big(\frac{\log^{1.5} (n+1)}{\e} + \frac{\log T}{\e}\Big) + \frac{1}{T} \Big(\frac{1}{\e}\log^{1.5} T + n\Big) \\
    &=O\Big(\frac{\log^{1.5} (n+1)}{\e} + \frac{\log T}{\e}\Big).
    \end{align*}
\end{proof}

\end{document}